\documentclass[12pt]{article}
\usepackage[margin=0.95in]{geometry}

\usepackage[english]{babel}
\usepackage[utf8x]{inputenc}
\usepackage{amsmath}
\usepackage{amssymb}
\usepackage{graphicx}
\usepackage{float}

\usepackage{mathtools}

\usepackage{natbib}
\usepackage{algorithm}
\usepackage{algpseudocode}

\usepackage{bm}
\usepackage{multirow}
\usepackage{comment}
\usepackage[toc,page]{appendix}
\usepackage{xcolor}
\usepackage{amsthm}

\newtheorem{prop}{Proposition}
\newtheorem{theorem}{Theorem}

\usepackage{color}
\usepackage{caption}
\usepackage{subcaption}

\usepackage{amsmath}
\DeclareMathOperator*{\argmax}{arg\,max}
\DeclareMathOperator*{\argmin}{arg\,min}

\title{\textbf{Where does the tail start? Inflection Points and Maximum Curvature as Boundaries}}

\author{Rafael Cabral$^1$, Maria de Iorio$^{1,2}$, Andrea Cremaschi$^3$}
\date{%
    $^1$Department of Paediatrics, Yong Loo Lin School of Medicine, National University of Singapore\\%
    $^2$Singapore Institute for Clinical Sciences, A*STAR\\%
    $^3$School of Science and Technology, IE University, Madrid\\[2ex]%
    \today
}

\begin{document}

\maketitle

\begin{abstract}

Understanding the tail behavior of distributions is crucial in statistical theory. For instance, the tail of a distribution plays a ubiquitous role in extreme value statistics, where it is associated with the likelihood of extreme events. There are several ways to characterize the tail of a distribution based on how the tail function, $\bar{F}(x) = P(X>x)$, behaves when $x\to\infty$. However, for unimodal distributions, where does the core of the distribution end and the tail begin? This paper addresses this unresolved question and explores the usage of delimiting points obtained from the derivatives of the density function of continuous random variables, namely, the inflection point and the point of maximum curvature. These points are used to delimit the bulk of the distribution from its tails. We discuss the estimation of these delimiting points and compare them with other measures associated with the tail of a distribution, such as the kurtosis and extreme quantiles. We derive the proposed delimiting points for several known distributions and show that it can be a reasonable criterion for defining the starting point of the tail of a~distribution.

%as the starting point of the tail

    %i.e. determining at which point $x$, the probability density function $\pi(x)$ behaves similarly to the tai, 
    
    %concept of tail of distribution not precisely defined one, give rough definition
    
    %asymptotic behaviour of the tail
    
    %Quantiles usually chosen

    %Alternative quantity, namely the point of maximum curvature of the pdf, and explore it's usage as the starting point of the tail
        
\end{abstract}


\begin{comment}
    
To do:
\begin{itemize}
    
    \item can be used to identify different extremal regions in the distribution, not well captured by kurtosis or extremal quantiles.
    
    \item usage of pmconv as a proxy for heavy tailedness

    \item another way to define tail, in terms of error between asymptotic approximation and ...

    \item quantiles vs delimiting points
    
    \item (in Section 2) plots derivatives of the Cauchy distribution and study it, mention that for other distributions defined on the real line it is different, mention exponential (one type), log-Gamma (another type).

\end{itemize}

\end{comment}

\section{Introduction}

%These are probability distributions whose tails are not exponentially bounded: that is, the tails decay slower than the exponential distribution, where 

In probability theory, the tails of a distribution are commonly studied in the context of the theory of heavy-tailed distributions \citep{foss2011introduction,taleb2020statistical}. Tail behavior is usually described by the tail function $\Bar{F}(x) = P(X>x)$ for  $x\to\infty$. For instance, a distribution is heavy-tailed if the tail function is not exponentially bounded, more precisely, $\lim_{x\to\infty} e^{tx}\Bar{F}(x) = \infty$, $\forall t>0$. There are two important subclasses of heavy-tailed distributions: long-tailed \citep{asmussen2003steady}, and subexponential distributions \citep{teugels1975class}. For an overview, see  \cite{foss2011introduction} and references therein. These distributions are often used to model the occurrence of extreme events in applied fields such as finance and climate science \citep{pisarenko2010heavy,longin2016extreme, cabral2023controlling}. Also, in robust statistics, heavy-tailed distributions can provide more robust inferences since they can reduce the sensitivity of the linear regression estimates with regards to outliers present in the data \citep{huber2004robust, cabral2023fitting}. The tails of the distribution are also important to describe the behavior of shrinkage priors in sparse regression \citep{carvalho2010horseshoe}.

%In the previous discussions of the tail of a distribution, the tail is not a precisely defined interval of the probability density function (pdf), in the sense that there is no specific point where the bulk of a distribution ends and the tail begins. 
Existing analyses of distribution tails primarily focus on their asymptotic behavior as $x$ approaches infinity, often overlooking what happens before infinity. Also, the tail is not a precisely defined interval of the probability density function (pdf), in the sense that there is no specific point where the bulk of a distribution ends and the tail begins. To our knowledge, there has yet to be a concrete discussion in the literature regarding the definition of a delimiting point for the left tail, $t_l$, or the right tail, $t_r$. This paper presents automatic approaches for selecting these delimiting points. We will concentrate our attention on continuous random variables with support on $\mathbb{R}$ or $\mathbb{R}^+$ and characterized by a unimodal pdf, where the delimiting points will be used to define the bulk or modal region, $[t_l, t_r]$, the proper tail interval $[t_r,\infty]$, and finally the left tail interval $[A,t_l[$ of the pdf, where $A$ can be 0 or $-\infty$ depending on the support of the distribution. These intervals can also be applied to other functions that describe the distribution, such as the tail function $\Bar{F}(x)$ or the hazard function $h(x) = f(x)/\Bar{F}(x)$.

%where $A$ and $B$ can be minus infinity or plus infinity, respectively

%the right tail of a distribution as the interval $[t_r,\infty]$ of the pdf, for some adequately chosen cut-off point $t_r$, and likewise for the left tail $[-\infty, t_l]$. These intervals can also be applied to other functions that describe the distribution, such as the tail function $\Bar{F}(x)$ or the hazard function $h(x) = f(x)/\Bar{F}(x)$.

%We consider a delimiting point for the left tail, $t_l$, and another for the right tail, $t_r$. For unimodal distributions with support on $[A, B]$, these delimiting points will be used to define the bulk or modal region $[t_l, t_r]$, the left tail interval $[A,t_l[$ and the right tail interval $[t_r, B]$, where $A$ and $B$ can be minus infinity or plus infinity, respectively. 

While the terms ``bulk" and ``tail" have not been precisely defined, they are often used informally to describe different distribution regions and convey an intuitive understanding of where the probabilities are concentrated. For unimodal distributions with support on the real line, the bulk or modal region of the pdf colloquially refers to a high-density interval centered around the mode. The bulk contains the values that are more likely to be observed. As we move towards the tails of the distribution, the probability density decreases, indicating that values in the tails are less likely to occur (see Figure \ref{fig:cauchy_intro}). 

\begin{figure}[h]
   \centering
\includegraphics[width=\linewidth]{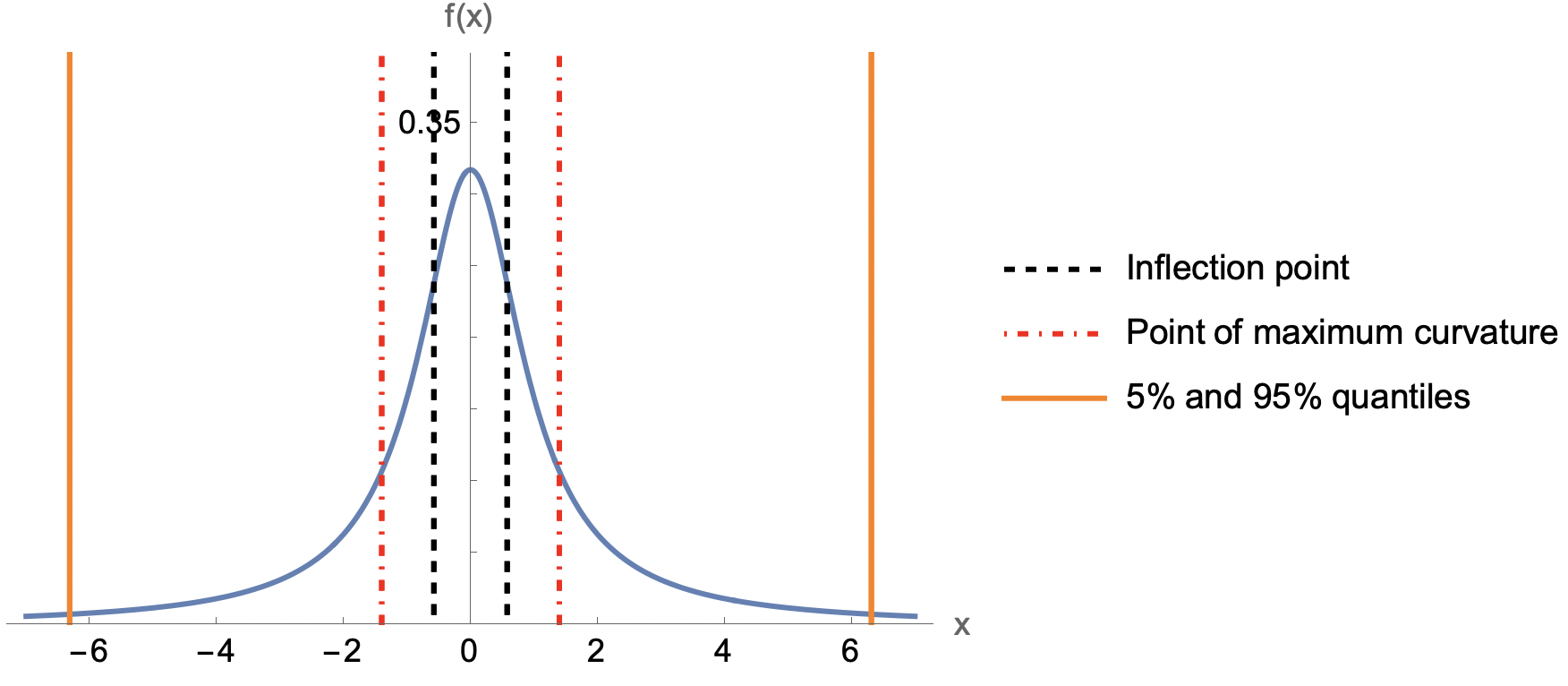}
\caption{Density function of Cauchy distribution (blue line). The inflection point, point of maximum curvature, and (5\% and 95\%) quantiles are marked as vertical dashed, dotted, and continuous lines, respectively.}
\label{fig:cauchy_intro}
\end{figure}
%\citep{vehtari2021rank}

The ``bulk" and ``tail" regions often correspond to different behaviors of the second derivative of the density function. For example, in the case of the Cauchy distribution illustrated in Figure \ref{fig:cauchy_intro}, the density function is concave for $|x|<1/\sqrt{3}$ and behaves like $(1-x^2)/ \pi + \mathcal{O}(x^3)$ near the origin. On the other hand, for $|x|>1/\sqrt{3}$, it becomes convex and behaves like $1/(\pi x^2) + \mathcal{O}(x^{-3})$ as $x$ approaches infinity. The inflection points $x \pm 1/\sqrt{3}$ mark the change from concavity to convexity.

%contributions
In this paper, we explore different methods that can be used to define modal regions and the tail's starting point based on the pdf's second derivative. The inflection point of the pdf is a natural candidate for many unimodal distributions since it marks the pdf's change from concavity to convexity. This will yield similar modal regions as in \cite{duong2008feature}, which focuses on regions where the pdf is significantly concave. Here, we propose using the point of the maximum curvature and contrast it with other measures related to heavy-tailedness, such as the kurtosis and extreme quantiles for commonly used distributions. A better understanding of the derivatives of density functions not only advances statistical theory but also opens doors for new methodologies for data analysis. For instance, an adequate identification of modal regions is useful in applications such as nonparametric clustering and bump hunting, as discussed in \cite{siloko2019note} and \cite{chacon2013data}.

% in the data in the context of feature significance, for instance, to decide whether features, such as local extrema, are statistically significant
%used the hessian of the density function (estimated through kernel density techniques) to identify modal regions in the pdf. The authors
In statistical data analysis, \cite{duong2008feature} develop a test for the null hypothesis that the hessian of the pdf is positive definite to identify significant concave regions. The method aims to identify multimodality and clusters in the data and is used to identify spatial hotspots in earthquake occurrences and high-density spots for large biological cells in flow cytometry studies. Another area where finding a transition point between the bulk and tail of the distribution is useful is in extreme value theory. It is common to exclude all but the most extreme observations when interest focuses on estimating only tail features, e.g., estimating the tail index $\alpha$\footnote{The tail index $\alpha$ characterizes tail functions with the asymptotic behavior $\bar{F}(x) = x^{-\alpha}L(x)$ as $x\to\infty$, where $L(x)$ is a slowly varying function.} or extrapolating to high quantiles from limited observed data. The Pickands-Balkema-de Haan Theorem \citep{balkema1974residual,pickands1975statistical}  justifies the peaks-over-threshold estimation method, which involves fitting a generalized Pareto distribution (GPD) to observations surpassing a specified threshold. Data-driven threshold selection methods are thus crucial and commonly rely on graphical diagnostic tools or automatic techniques \citep{scarrott2012review}. However, determining the optimal threshold presents challenges, for instance, due to potentially multiple transition regimes to the tail. Various approaches have emerged to estimate the entire density function by employing a mixture model, which considers a separate model for the bulk and a GPD model for the tail, with examples including those proposed by,  \cite{tancredi2006accounting}, \cite{macdonald2011flexible}, and \cite{do2012semiparametric}. However, these methods often rely on heuristic techniques to find the threshold with limited mathematical analysis support.

%, aiming to strike a balance between the bias induced by the asymptotic tail approximation and parameter estimation uncertainty due to an often small number of excess data above the threshold

%This paper significantly contributes by exploring the fundamental concept of curvature and its applications in probability distributions. It introduces the definition of curvature and the construction of crucial delineating points, like the point of inflection and maximum curvature, between the distribution's bulk and tail. The exploration of properties, examination of existence and interrelationships of these points, and the introduction of a study on maximizers represented by $\theta^r$ add depth to distribution understanding. The paper also provides an estimator for maximizers, discusses its uniform consistency, and presents a simulation study highlighting challenges in estimating PMCurv and PInf. Applying these methods to study well-known distributions, including the log-Normal, exponential, Gaussian, Student's $t$-distribution, and skew-$t$ distribution, contributes to a comprehensive understanding of curvature-based delineation across diverse probability structures. Overall, this paper offers a novel perspective on curvature with practical implications for researchers and practitioners.

The paper is organized as follows. In Section \ref{sect:curvature}, we define the curvature of a pdf, the delimiting points between the bulk and tail of the distribution that can be constructed from it, and examine their properties, such as their existence and how they are related. Section \ref{sect:pmcurv} addresses the estimation of these delimiting points from the data, where we provide sample versions of these points and prove their uniform consistency. In Section \ref{sect:examples}, we derive the delimiting points for several known distributions. Lastly, in Section \ref{sect:discussion}, we discuss the main results.% and provide some directions for future work. 

%discuss usage as pmcurv as heavy-tailedness measure of sorts... empirical links etc...
%mention we focus only on continuous distributions, most defined either on the real line or positive reals

%In many problems, there are physical solutions , ocurrence of extreme winds that are destructive, however there are cases where it is useful for defining starting point of tail (PC priors)

%MAKE IT MORE EXPLORATORY, DONT SAY PMCURV IS THE BEST...
%ADD PROOF THAT THERE IS A PMCURV AND PMINF FOR UNIMODAL DISTS

\section{Curvature} \label{sect:curvature}

% why curvature? where does it come from?
The concept of curvature is fundamental in mathematics and is widely explored in various fields. In this section, we define the curvature and, based on it, construct useful delimiting points between the bulk and tail of the distribution, namely the point of inflection and maximum curvature. We also present some properties of these delimiting points and examine their existence and how they are related to one another.

Consider the parametric representation $\gamma(t)=(x(t), y(t))$ of a plane curve, assumed to be twice differentiable. We assume the derivative $d\gamma/dt$ is well-defined, differentiable, and not identical to the zero vector across the parametrization domain. Utilizing this parametrization, the signed curvature can be expressed as

$$
\kappa(t)=\frac{x^{\prime}(t) y^{\prime \prime}(t)-y^{\prime}(t) x^{\prime \prime}(t)}{\left((x^{\prime}(t))^2+(y^{\prime}(t))^2\right)^{3 / 2}}
$$
with primes denoting derivatives with respect to $t$. Intuitively, the curvature is the amount by which a curve deviates from being a straight line. Our interest, pertaining to the establishment of a probability density function tail starting point, lies in the graph of a function $y=f(t)$, where $f(t)$ is the density function of a given random variable $X$. This is a specific instance of a parameterized curve defined as
$$
\begin{aligned}
& x(t)=t \\
& y(t)=f(t)
\end{aligned}
$$
Given that the first and second derivatives of $x$ with respect to $t$ are 1 and 0, respectively, the above formula can be simplified to
\begin{equation}\label{eq:curv}
\kappa(t)=\frac{f^{\prime \prime}(t)}{\left(1+f^{\prime}(t)^2\right)^{3 / 2}}
\end{equation}

%For distributions with support on the real line, we only study the right tail and analyse the curvature $\kappa(t)$ when $t>0$.

\subsection{Point of inflection and maximum curvature }

The graph of the differentiable function has an inflection point at $(x, f(x))$ if and only if its first derivative $f'$ has an isolated extremum at $x$. In our context, the inflection point is given by \mbox{$\text{PInf} = \argmax_t |f^{\prime}(t)|$} and can also be found by computing the roots of the second derivative of the pdf. On the other hand, the point of maximum curvature is the point $t$ where the curvature $\kappa(t)$ in \eqref{eq:curv} achieves its highest value: $$\text{PMCurv} = \argmax_t \kappa_X(t).$$
In general, there is no closed-form expression for $\text{PInf}$ and $\text{PMCurv}$, and they need to be computed using numerical optimization algorithms. 

To simplify calculations, the curvature in \eqref{eq:curv} can be approximated by the second derivative of the pdf when the derivative in the denominator is much smaller than 1 in absolute value. Indeed, we have:
$$
k(x)=f''(x)\left(1+O\left(f'(x)^2\right)\right)
$$
In the context of probability distributions, this can happen when we apply a scale transformation $X \to \sigma X$ and $\sigma$ is sufficiently large, since $\pi_{\sigma X}(t) = \pi_{X}(t/\sigma)/\sigma$, and $\pi_{\sigma X}'(t) = \pi_{X}'(t/\sigma)/\sigma^2$ which goes to 0 when $\sigma \to \infty$. Intuitively, the densities become flatter when $\sigma$ increases, with the first derivative going to zero. In this case, $\kappa_X(t) \approx \pi_X^{\prime \prime}(t)$ and the point of maximum curvature can be approximated by the point of maximum convexity (PMConv), which is defined as 
\begin{equation} \label{eq:approximation}
    \text{PMCurv} \approx \text{PMConv} = \argmax_t f^{\prime \prime}(t)
\end{equation}
We now study the behavior of $\text{PInf}_X$, $\text{PMConv}_X$ and $\text{PMCurv}_X$ of a random variable $X$ when applying a location-scale transformation. We have that $\text{PInf}_{\mu+\sigma X} = \mu+\sigma \text{PInf}$, $\text{PConv}_{\mu+\sigma X} = \mu+\sigma \text{PMConv}_{X}$ and  $\text{PMCurv}_{X+\mu} = \text{PMCurv}_{X} + \mu$. On the other hand, in general, $\text{PMCurv}_{\sigma X} \neq \sigma \text{PMCurv}_{X}$. However, under the approximation in equation \eqref{eq:approximation}, which holds for large $\sigma$, we have 

\begin{equation}\label{eq:approx2}
  \text{PMCurv}_{\sigma X} \approx \text{PMConv}_{\sigma X} =  \sigma \text{PMConv}_{X}  
\end{equation}

In Figure~\ref{fig:spmc1}, we show the PMCurv and PMConv for several continuous random variables. In Sections~\ref{sect:pmcurv} and \ref{sect:examples}, we consider the approximation $\text{PMCurv} \approx \text{PMConv}$.
%We show in Figure \ref{fig:spmc1}, the PMCurv and PMConv for several continuous random variables, and in Figure \ref{fig:spmc2} the corresponding quantiles. 
%mention this approximation is typically used

Note that  for unimodal distributions with mode at $\theta$, $f''(x)$  typically possess two local maximums, one at $x<\theta$ relating to the left tail, $\text{PMConv}_l$, and another at $x>\theta$, relating to the right tail, $\text{PMConv}_r$ (see Figure \ref{fig:cauchy_intro}). For symmetric unimodal distributions with zero mean $\text{PMConv}_l = -\text{PMConv}_r$, but generally this is not the case (for example, see the Log-Normal distribution in Section \ref{eq:lognormal}). Thus, it is useful to define 

\begin{equation}\label{eq:sample1}
 \text{PMConv}_l = \argmax_{t<\theta} f^{\prime \prime}(t) \ \ \text{and} \ \ \text{PMConv}_r = \argmax_{t>\theta} f^{\prime \prime}(t)   
\end{equation}
where we assume $f''$ has an isolated maximum and root on each side. Similarly, we define the inflection point associated with the left and right tail:
\begin{equation}\label{eq:sample2}
\text{PInf}_l  = \argmax_{t<\theta} |f^{\prime}(t)| \ \ \text{and} \ \ \text{PInf}_r = \argmax_{t>\theta} |f^{\prime}(t)|
\end{equation}

%for several random variables with support on $\mathbb{R}^+$, such as the Log-Gamma random variable, the curvature of the pdf has several local maxima (see Fig. \ref{fig:lognormalpdf}). Therefore, to construct practical intervals for the tails and bulk of the distribution, we first compute the several local maxima of the curvature and then pick the ones with the smallest and largest argument as cutoff points $t_l$ and $t_r$, respectively. %argmax is at 0, careful for this distributions, that have unbounded second derivatives 

Moreover, if $\text{PMConv}_r=0$ for a random variable $X$ with support on $\mathbb{R}^+$, then for large enough $\sigma$,$\text{PMCurv}_{\sigma X}$ is also equal to 0 (see equation \ref{eq:approx2}). This happens, for instance, for the exponential, Pareto, Weibull (for rate parameter $\beta=1$), and log-Gamma distributions (see Figure~\ref{fig:spmc1}). The density functions of these distributions have a singularity at 0, as well as their second derivative (thus $\text{PMConv}_r=0$). Furthermore, the second derivative is always positive, and thus there are no inflection points either. For such distributions, which have a mode at 0, the delimiting points we are considering are not useful in the sense of not providing a practical delimiting point between the modal region and the tail of the distribution.

%Distributions can then be subdivided into two: those that have $\text{PMConv}=0$, and those where $\text{PMConv}>0$.
Finally, it is also useful to work with $F(\text{PMConv}_r)$ and $F(\text{PInf}_r)$, the cdf of the distribution evaluated at $\text{PMConv}_r$ and $\text{PInf}_r$ (and likewise for $\text{PMConv}_l$ and $\text{PInf}_l$).  Unlike $\text{PMConv}_r$, which changes under a scale and location transformation, we can use $F(\text{PMConv}_r)$ as a measure that is invariant under scale and location transformations, which facilitates interpretation when comparing different distributions (see Figure~\ref{fig:spmc1} and  \ref{fig:spmc2}). For instance, this measure is about 0.9584 for the Gaussian distribution, while for the Cauchy distribution, it is equal to 0.75. As we will see in Section \ref{sect:examples}, $F(\text{PInf}_r)$ and $F(\text{PConv}_r)$ tend to decrease as the distribution becomes more heavy-tailed for several families, such as the Student's $t$-distribution, which encompasses the Gaussian and Cauchy distributions. We also observe that for the Gaussian distribution, $\text{PConv}_l$ and  $\text{PConv}_r$ are very close to the 5\% and 95\% quantiles, which are commonly used to define extreme values or outliers. %We will see in Section \ref{sect:examples} that $F(\text{PMConv}_r)$ tends to decrease for heavier-tailed distributions.

\begin{figure}[H]
   \centering
   \includegraphics[width=\linewidth]{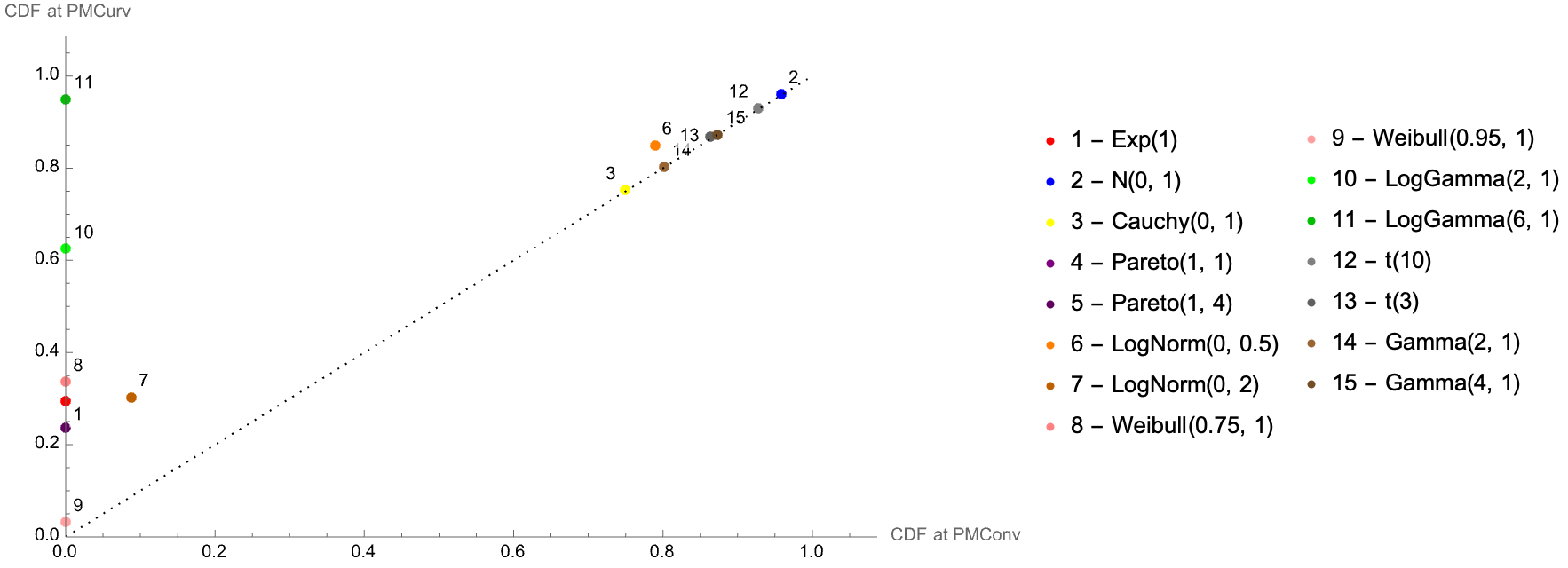}
   \caption{Cdf of $\text{PMCurv}_r$ against cdf of $\text{PConv}_r$ for several distributions.} %maybe 
  \label{fig:spmc1} 
\end{figure}

\begin{figure}[H]
   \centering
   \includegraphics[width=\linewidth]{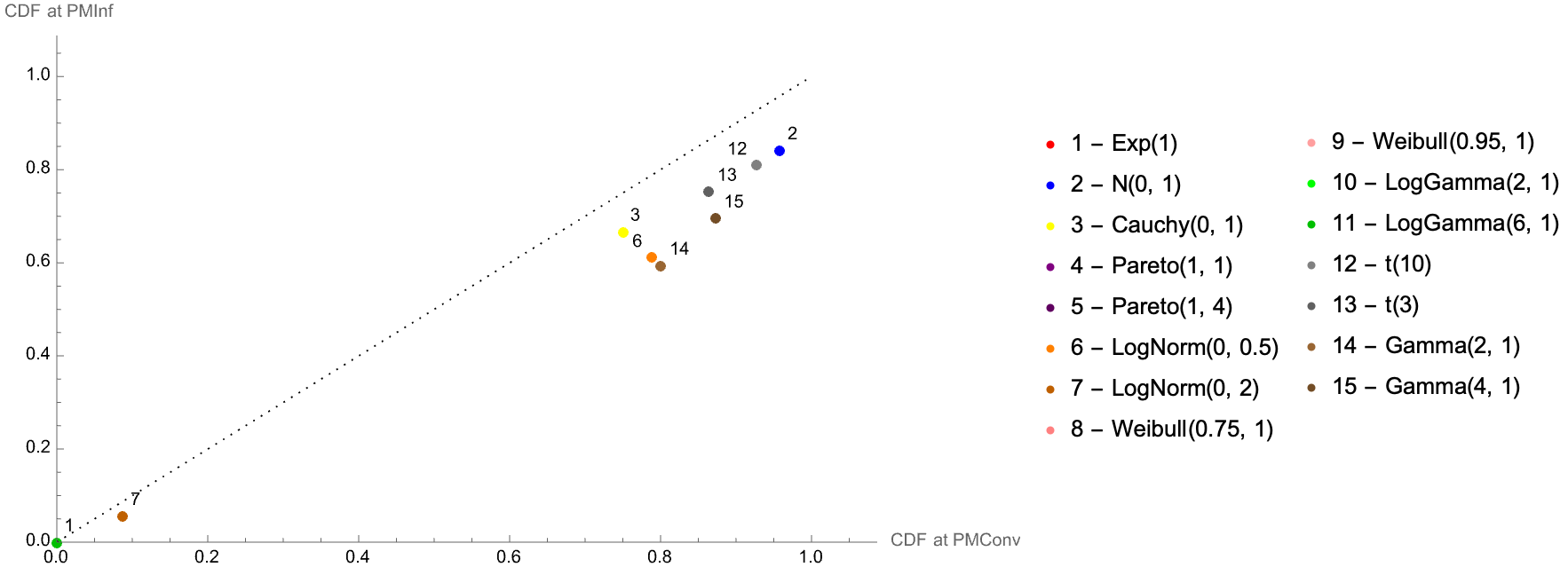}
   \caption{Cdf at $\text{PInf}_r$ against cdf at $\text{PConv}_r$ for several distributions. Many of the distributions do not have $\text{PInf}_r$, and for those, we omit the results. } %maybe 
  \label{fig:spmc2} 
\end{figure}

We now discuss the existence of the delimiting points introduced so far and prove their existence in the case of unimodal distributions. %s some conditions that guarantee

\begin{theorem}\label{theo:1}
 Let $f(x)$ be the unimodal density function of a continuous random variable with support on the real line, with mode $\theta$, and let $f^{(r)}(x)$ be uniformly continuous for $r=0,1,2$.
 Then, $f(x)$ has at least one inflection point and one point of maximum convexity for $x>\theta$, and similarly for $x<\theta$. 
\end{theorem}
\begin{proof}
  The proof is given in Appendix \ref{theoproof:1}.
    
\end{proof}

For unimodal distributions with support on $[C,\infty]$ and mode in $\theta>C$, the same argument can be used to show that an inflection point and PMConv exist for $x>\theta$. When the unique mode is at $C$, the boundary of the support, we further require that $f''(C)<0$; otherwise, these delimiting points may not exist for $x>C$. An example is the exponential distribution, which has a positive second derivative for $x\geq0$, with a maximum at $x=0$, and therefore, no inflection point, and the point of maximum convexity is at 0. The following proposition indicates that the modal region $]\text{PInf}_l,\text{PInf}_r[$ is smaller than $]\text{PMConv}_l,\text{PMConv}_r[$.

%I can extend this without requiring unimodal condition
\begin{prop}\label{prop:1}
Under the conditions of Theorem \ref{theo:1}, if,  to the right of the mode, there is only one inflection point, $\text{PInf}_r$, and one point of maximum convexity $\text{PMConv}_r$, then  $\text{PMConv}_r>\text{PInf}_r$. Similarly, at the left of the mode, $\text{PMConv}_l<\text{PInf}_l$.    
\end{prop}
\begin{proof}
  The proof is given in Appendix \ref{propproof:1}.
\end{proof}
%I can extend this without requiring one pair of points

\section{Estimation}  \label{sect:pmcurv}
The inflection points and points of maximum convexity are related to the extremums of the first and second derivatives of the pdf, respectively. Therefore, we need to study

\begin{equation}\label{eq:maximiser}
    \theta^r = \argmax_{-\infty<x<\infty} f^{(r)}(x)
\end{equation} 
We present an estimator for these maximizers and discuss their uniform consistency. Finally, we present a simulation study and highlight some difficulties in estimating PMCurv and PInf.

For the inflection point, we are also interested in computing $\argmin_x f'(x)$, and it is straightforward to extend the arguments below to this case. Also, for the unimodal distributions of Section \ref{sect:examples}, there are two local maximums for the first and second derivatives that we are interested in, one pertaining to the left tail and the other to the right tail, and this case, we can find unique maximizers for each tail by constraining the optimization region of \eqref{eq:maximiser} to be at the right or left of the mode.% The arguments below can also be extended to this case as well, where one would consider the optimization for the domains $x<\theta$ and $x>\theta$, where $\theta$ is the mode of the distribution.
% $x = \argmax_y |f'(y)|$
% Namely, $x$ is an inflection point if $f'(x)$ has an isolated extremum at $x$.

The estimation of $\theta^r$ is ultimately related to the estimation of the probability density function $f(x)$ and its derivatives. Let $X_1, \dotsc, X_n$ be independent and identically distributed with distribution function $F(x) = \int_{-\infty}^x f(u) du$. \cite{rosenblatt1956remarks} considered nonparametric estimators of the form

$$
f_n(x)=\frac{1}{n} \sum_{i=1}^n \frac{1}{h_n} K\left(\frac{x-X_i}{h_n}\right)
$$
where the kernel function $K(x)$ is a probability density function and $h_n$ is the bandwidth which goes to 0 as $n$ increases. Usually, $K(x)$ satisfies the following conditions

$$
\int K(x) d x=1, \\
\int x K(x) d x=0 \text { and } \\
\int x^2 K(x) d x \neq 0
$$

Properties of this estimator, including uniform consistency and asymptotic normality, are derived in \cite{parzen1962estimation}. \cite{schuster1969estimation} studies the uniform convergence of the derivatives of $f_n(x)$,

$$
f_n^{(r)}(x)=\frac{1}{n} \sum_{i=1}^n \frac{1}{h_n^{r+1}} K^{(r)}\left(\frac{x-X_i}{h_n}\right)
$$
The choice of bandwidth $h_n$  is discussed in \cite{hardle1990bandwidth} and \cite{politis2015adaptive}. Sample versions of the points of inflection and maximum convexity can then be obtained by replacing the pdf $f$ in  \eqref{eq:sample1} and \eqref{eq:sample2} by its estimator $f_n$.

We assume $\int |u|K(u)du$ is finite and $K^{(r)}(x)$ is a continuous function of bounded variation. These conditions are satisfied, for instance, when $K(x)$ is the density function of a standard Normal distribution. We further assume that $f^{(r)}(x)$ is uniformly continuous and possesses a maximiser $\theta^{r}$ defined by

$$
f(\theta^r) = \max_{-\infty<x<\infty}f^{(r)}(x)
$$
and that this maximizer is unique. An estimator for $\theta^r$
is given by
\begin{equation*}\label{eq:thetanr}
    \theta_n^r = \argmax_{-\infty<x<\infty}f_n^{(r)}(x)
\end{equation*}
and in the following theorem, we show the consistency of $\theta_n^r$ as an estimator of $\theta^r$.

\begin{theorem}\label{theo:2}
Consistency of $\theta_n^r$ as an estimator of  $\theta^r$. If $h_n$ is a function of $n$ satisfying:
$$
\lim_{n\to \infty} n h_n^{2r+2} = \infty
$$
then for every $\epsilon>0$
$$
P(|\theta_n^{r}-\theta^r|>\epsilon) \to 0 \ \ \ \ \text{as} \ n \to \infty
$$
\end{theorem}
\begin{proof}
The proof is given in Appendix \ref{theoproof:2}.
\end{proof}

%Examples of choices for $K(x)$ are the density function of the standard normal distribution and the density function of a uniform random variable with support between $-1$ and $1$. The sample point of maximum convexity is then given by:

%$$
%\text{PMConv}_n(x) = \argmax f_n^{(2)},
%$$
%and similar sample definitions apply for the other measures that involve density derivatives, such as the point of maximum curvature.

\begin{comment}
There is an alternative way to compute the second derivative based on the finite difference expression of the third derivative of the distribution function:

\begin{align*}
    f^{(2)}(x) = F^{(3)}(x) &\approx \frac{F(x+2h_n)-2F(x+h)+2F(x-h_n)-F(x-2h_n)}{2h^3} \\
                     &= \frac{P(x-2h_n<X<x+2h_n) - 2P(x-h_n<X<x+h_n)}{2h_n^3},
\end{align*}

An estimator derived from the previous relationship is then given by:
$$
  f_{n,\star}^{(2)}(x) = \frac{1}{nh^3}\sum_{i=1}^n \gamma\left(\frac{x-X_i}{h}\right),
$$

where,

$$
\gamma(x) = 
\begin{cases}
 1/2, &  1 \leq |x| < 2 \\
 -1/2, &  |x|<1 .
\end{cases}
$$

The sample point of maximum convexity is then given by:

$$
\text{PMConv}_n(x) = \argmax f_n^{(2)},
$$

\end{comment}

\subsection{Simulation study}  

We utilize the Gaussian kernel because it has derivatives of all orders and simplifies the required mathematical computations. The derivatives are expressed as \(K^{(r)}(x)=(-1)^r H_r(x) K(x)\), where \(H_r(x)\) represents the \(r\)th Hermite polynomial. The initial five Hermite polynomials are denoted as \(H_0(x)=1, H_1(x)=x, H_2(x)=x^2-1, H_3(x)=x^3-3x\), and \(H_4(x)=x^4-6x^2+3\). Therefore, an estimator for the \(r\)th derivative of the density function is given by
\[
f_n^{(r)}(x)=\frac{(-1)^r}{\sqrt{2 \pi} n h^{r+1}} \sum_{i=1}^n H_r\left(\frac{x-X_i}{h}\right) \exp\left({-\frac{1}{2}\left(\frac{x-X_i}{h}\right)^2}\right)
\]
We chose $h_n$ to minimize the asymptotic mean integrated squared error (AMISE), as detailed in \cite{guidoum2020kernel} and \cite{siloko2019note}. Namely, we have:
\begin{equation}\label{eq:hAMISE}
h_n^r = \left[\frac{(2 r+1) R\left(K^{(r)}\right)}{\mu_2(K)^2 R\left(f^{(r+2)}\right)}\right]^{\frac{1}{2 r+5}} \times n^{-\frac{1}{2 r+5}}
\end{equation}
where $R(g) = \int g(x)^2 dx$. There are several estimators available for $R\left(f^{(r+2)}\right)$, such as those discussed in \cite{guidoum2020kernel}. However, in this simulation study, we utilize the true value $R\left(f^{(r+2)}\right)$ because the density function $f$ of the simulated data is known.

We simulate $n$ observations from a Student's $t$-distribution  with $\nu$ degrees of freedom and consider the simulation parameters $n \in \{100,500,2000\}$ and $\nu \in \{1,5,100\}$. For each configuration, we repeat the simulation $N=1000$ times and compute the sample inflection point and sample point of maximum convexity of the right tail using $\argmax_{x>0} - f_n'(x)$ and $\argmax_{x>0} f_n''(x)$, respectively. The mean squared errors (MSE) are given in Tables \ref{table:1} and \ref{table:2}, and, as expected, the MSE decreases with $n$ and is larger for the sample version of $\text{PMConv}_r$ compared to $\text{PInf}_r$ as the former involves higher derivatives. The MSE also decreases with $\nu$ since, as the distributions become closer to a Gaussian, the derivatives $f'(x)$ and $f''(x)$ fluctuate less near the origin and can be more accurately estimated by kernel density methods.

Figures~\ref{fig:kde0}, \ref{fig:kde1}, and \ref{fig:kde2} show the estimated density functions and their first two derivatives for simulated data with parameters $n=2000$ and $\nu=1$. We can see in Figure~\ref{fig:kde2} that the estimates over-smooth the second derivative, and to better capture the true shape of the second derivative near 0 and the PMCurv, we would need a smaller bandwidth than the one given by \eqref{eq:hAMISE}. A potential solution is offered by implementing adaptive bandwidth techniques \citep{politis2015adaptive} where the bandwidth depends on the location. % used unbiased least square cross validation \cite{bowman1984alternative,rudemo1982empirical}

\begin{table}
\centering
\begin{tabular}{cccc}
         & $\nu=1$ & $\nu=5$ & $\nu=100$ \\ \hline
$n=100$  & 0.245   & 0.181   & 0.189     \\
$n=500$  & 0.113   & 0.076   & 0.062     \\
$n=2000$ & 0.054   & 0.046   & 0.035    
\end{tabular}
\caption{MSE for the estimated inflection points considering $n$ simulated observations from a Student's $t$-distribution with $\nu$ degrees of freedom. }
\label{table:1}
\end{table}

\begin{table} 

\centering
\begin{tabular}{cccc}
         & $\nu=1$ & $\nu=5$ & $\nu=100$ \\ \hline
$n=100$  & 1.144   & 0.893   & 0.245     \\
$n=500$  & 0.689   & 0.494   & 0.339     \\
$n=2000$ & 0.373   & 0.297   & 0.192    
\end{tabular}
\caption{MSE for the estimated points of maximum convexity considering $n$ simulated observations from a Student's $t$-distribution with $\nu$ degrees of freedom.}
\label{table:2}
\end{table}

%use default h In package, see if it improves
% we could 

\begin{figure}[H]
   \centering
   \includegraphics[width=0.65\linewidth]{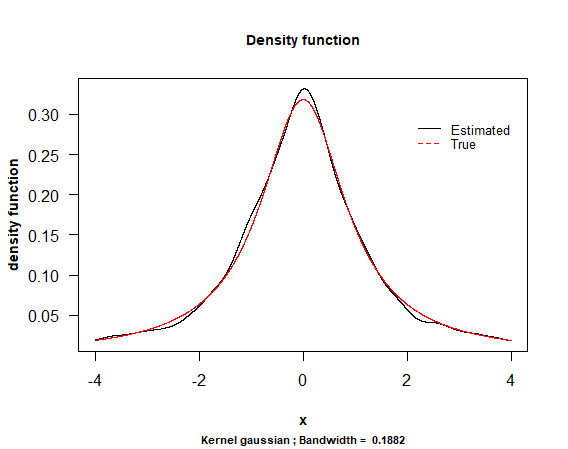}
   \caption{Density function of a Cauchy distribution estimated from 2000 simulated observations. } %maybe 
  \label{fig:kde0} 
\end{figure}

\begin{figure}[H]
   \centering
   \includegraphics[width=0.65\linewidth]{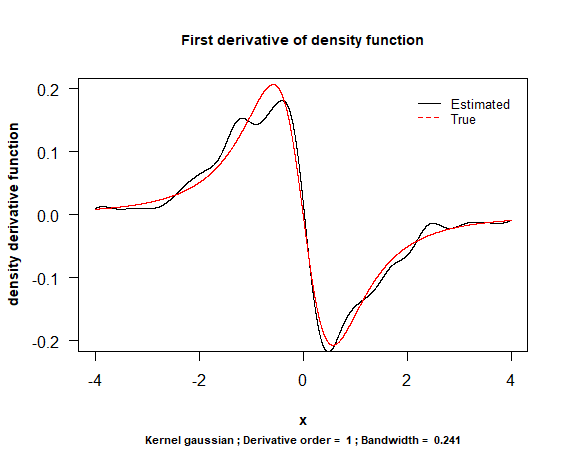}
   \caption{First derivative of the density function of a Cauchy distribution estimated from 2000 simulated observations.}
  \label{fig:kde1} 
\end{figure}

\begin{figure}[H]
   \centering
   \includegraphics[width=0.65\linewidth]{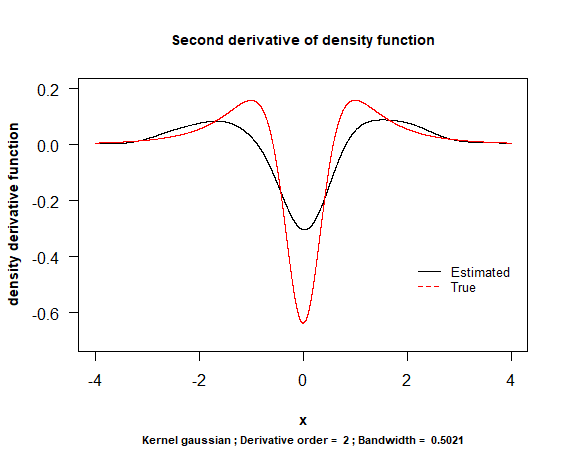}
   \caption{Second derivative of the density function of a Cauchy distribution estimated from 2000 simulated observations} %maybe 
  \label{fig:kde2} 
\end{figure}

\section{Examples} \label{sect:examples}

In this section, we investigate the delimiting points of several known distributions based on the methods described in Section \ref{sect:curvature}. Following the discussion in Section \ref{sect:curvature}, we consider distributions with support on $\mathbb{R}^+$ with mode larger than 0 (Log-Normal distribution), with mode at 0 (exponential distribution) and several unimodal distributions with support on $\mathbb{R}$ (Gaussian, Student's $t$-distribution and Skew-$t$ distribution). We compare these delimiting points with measures of heavy-tailedness and asymmetry, namely the kurtosis and extreme quantiles.

% , the stable distributions).

\subsection{Log-Normal distribution}\label{eq:lognormal}

%The Log-Normal distribution has support on $\mathbb{R}^{+}$, and for the parameters $\mu=0$ and $\sigma=0.5$, the curvature $\kappa(x)$ is maximized when $x$ is smaller than the mode, and so the PMCurv pertains to the left tail (see Figure \ref{fig:lognormalpdf}). For defining a useful cutoff point for the right tail, we compute the largest argument where the curvature $\kappa(x)$ has a local maximum. A similar reasoning applies when computing the inflection points and points of maximum convexity for the right tail. 

%and \ref{fig:lognormalpdf2}
Figure~\ref{fig:lognormalpdf} shows the density function of the Log-Normal distribution, its second derivative, and the delimiting points. The point of maximum convexity for the left tail is $\text{PMConv}_l = e^{\mu - 2\sigma^2-\sigma\sqrt{3+\sigma^2}}$ and for the right tail is $\text{PMConv}_r = e^{\mu - 2\sigma^2+\sigma\sqrt{3+\sigma^2}}$. These delimiting points are used to define the modal region $[\text{PMConv}_l, \text{PMConv}_r]$, the left tail interval $[0,\text{PMConv}_l[$ and the right tail interval $[\text{PMConv}_r,\infty]$. Likewise, the modal region based on the inflection points is given by:

$$[e^{\frac{1}{2} \left(2 \mu -3 \sigma ^2-\sigma\sqrt{\sigma ^2+4}  \right)}, e^{\frac{1}{2} \left(2 \mu -3 \sigma ^2+\sigma\sqrt{\sigma ^2+4}  \right)}]$$

%When $x\to \infty$, the density behaves like:
%$$
%\left(\frac{1}{\sqrt{2 \pi } \sigma  x}+O\left(\frac{1}{x}^3\right)\right) e^{-\frac{(\mu -\log (x))^2}{2 \sigma ^2}+O\left(\left(\frac{1}{x}\right)^3\right)}.
%$$
The larger $\sigma$, the slower the tails of the density function decay, and the higher the kurtosis is ($\text{Kurt} = 3 e^{2 \sigma ^2}+2 e^{3 \sigma ^2}+e^{4 \sigma ^2}-3$). On the other hand, the modal regions defined by the inflection point and point of maximum convexity shrink, and the delimiting points for the tails get closer to the mode. Moreover, the 95\% quantile is given by $e^{\mu +1.64485 \sigma }$, and thus, if one uses the quantiles, the modal region widens as we increase $\sigma$, having the opposite behavior as the modal region defined by the previous delimiting points (see Figure \ref{fig:ln2}).

\begin{figure}[H]
   \centering
\includegraphics[width=\linewidth]{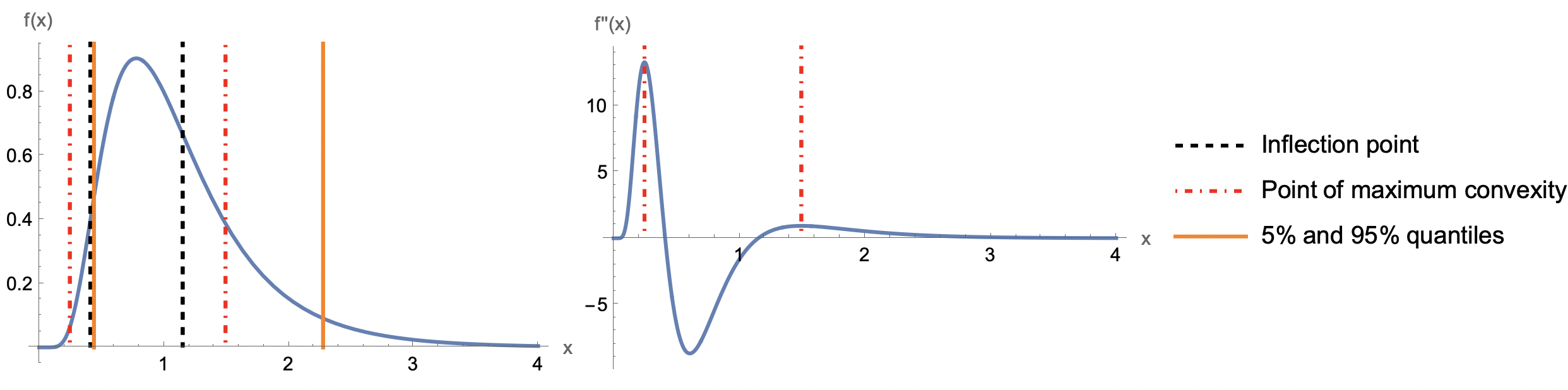}
\caption{Density function and its second derivative of a Log-Normal distribution with parameters $\mu=0, \ \sigma=0.5$. The black, red, and orange vertical lines are drawn in correspondence to the $\text{PInf}$, $\text{PMConv}$, and the 5\% and 95\% quantiles, respectively. }
\label{fig:lognormalpdf}
\end{figure}

%\begin{figure}[H]
%   \centering
%\includegraphics[width=\linewidth]{Images/pdf52.png}
%\caption{Density function and its second derivative of a Log-Normal distribution with parameters $\mu=0, \ \sigma=1$. The black, red, and orange vertical lines are the $\text{PInf}$, $\text{PMConv}$, and the 5\% / 95\% quantiles, respectively. }
%\label{fig:lognormalpdf2}
%\end{figure}

\begin{figure}[H]
   \centering
\includegraphics[width=\linewidth]{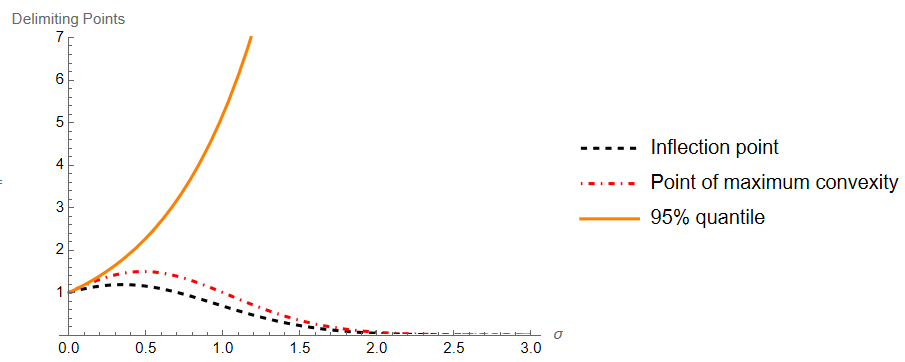}
\caption{Delimiting points for the right tail as a function of $\sigma$ for a Log-Normal distribution with parameters $\mu=0$ and $\sigma$.}
\label{fig:ln2}
\end{figure}

%\begin{figure}[H]
%   \centering
%\includegraphics[width=0.9\linewidth]{Images/pmc5.png}
%\caption{$\text{PMCurv}$ and 95\% quantile of a Log-Normal distribution for $\mu=0$ and varying scale parameter.}
%\end{figure}

\subsection{Exponential distribution}

As shown in Figure~\ref{fig:exp1}, the density of the Exponential distribution has a singularity at 0 and no inflection points. For this distribution, $\text{PMConv}=0$ since the second derivative of the pdf has a maximum at 0, but $\text{PMCurv}=\log(2\lambda^6)/(2\lambda)$ for $\lambda>2^{-1/6}$, and is 0 otherwise, where $\lambda$ is the rate parameter. Thus, it would be more appropriate to define the modal region $[0,t_r]$ and right tail $[t_r,\infty]$ in terms of the quantiles (see Figure~\ref{fig:exp2}).

\begin{figure}[H]
   \centering
\includegraphics[width=\linewidth]{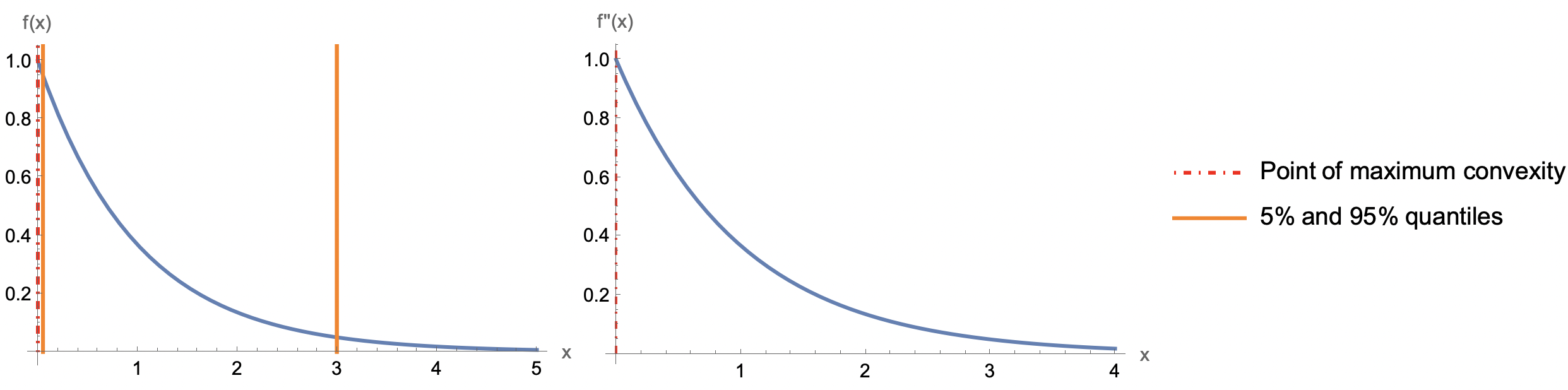}
\caption{Density function and its second derivative of an Exponential distribution with rate parameter 1. The red and orange vertical lines are drawn in correspondence with the $\text{PMConv}$ and the 5\% and 95\% quantiles, respectively.}
\label{fig:exp1}
\end{figure}

\begin{figure}[H]
   \centering
\includegraphics[width=\linewidth]{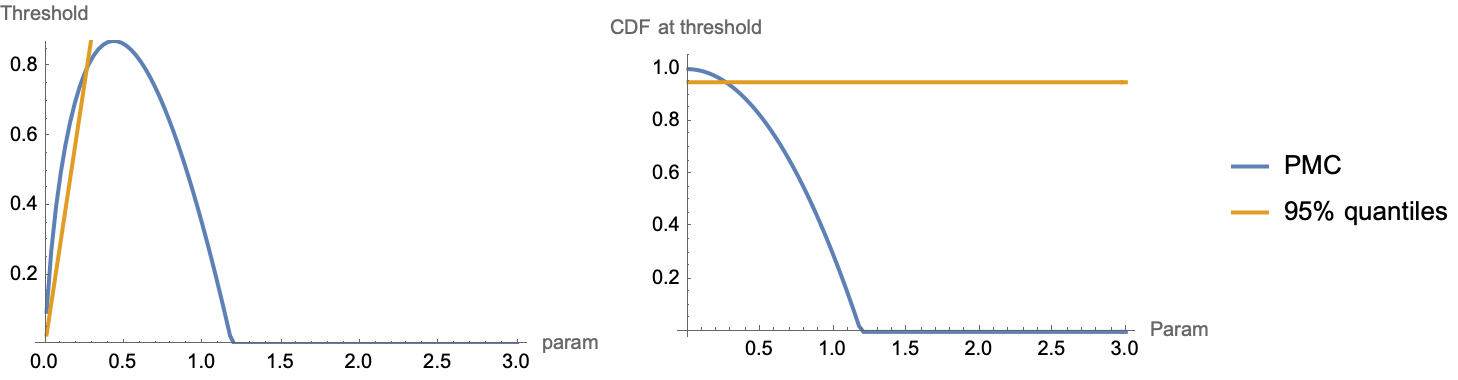}
\caption{$\text{PMCurv}$ and 95\% quantile of an exponential distribution with varying scale parameter.}
\label{fig:exp2}
\end{figure}

\subsection{Gaussian distribution}

Figure \ref{fig:gauss1} shows the density of the Gaussian distribution, where the inflection points are exactly one standard deviation from the mean: $\mu \pm \sigma$ ($\text{PMInf}_r$ corresponds to the quantile 0.841). Now, the points of maximum convexity are located at $\sqrt{3}\sigma$ standard deviations from the mean: $\mu \pm \sqrt{3}\sigma$ ($\text{PMConv}_r$ corresponds to quantile 0.958). There is no closed-form expression for the PMCurv, although, for large $\sigma$, it tends to $\text{PMConv} = \mu \pm \sqrt{3}\sigma$, as can be seen in Figure \ref{fig:gauss2}. 

%We can observe in Figure \ref{fig:PMcurv}, how the point of maximum curvature changes with the scale parameter $\sigma$, where for small scales 

\begin{figure}[H]
   \centering
\includegraphics[width=\linewidth]{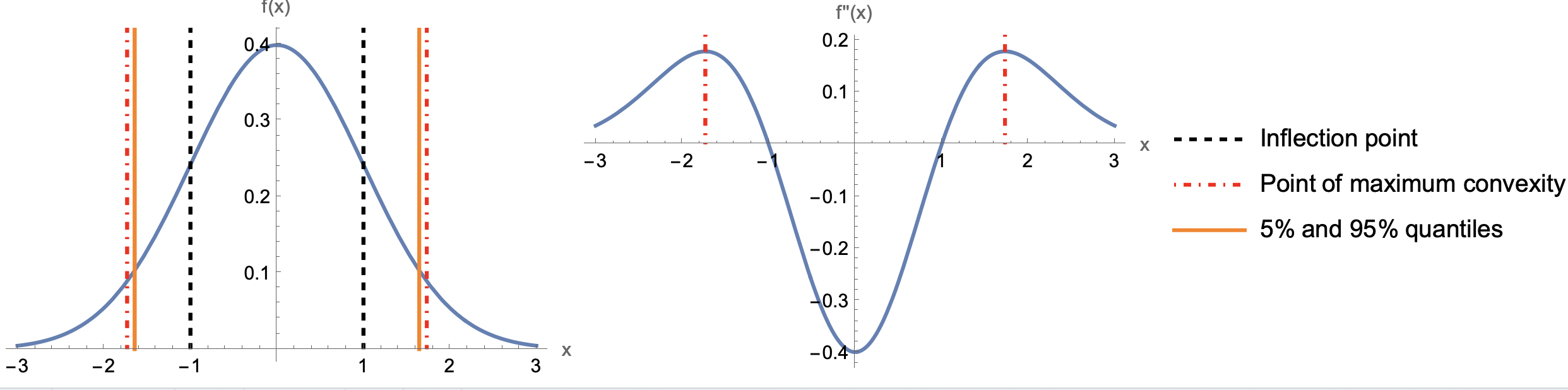}
\caption{Density function and its second derivative of a standard Normal distribution. The black, red, and orange vertical lines are drawn in correspondence with the $\text{PInf}$, $\text{PMConv}$, and the 5\% and 95\% quantiles, respectively.}
\label{fig:gauss1}
\end{figure}

\begin{figure}[H]
   \centering
\includegraphics[width=0.9\linewidth]{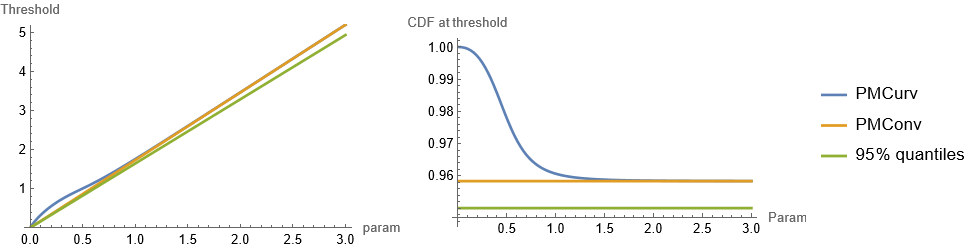}
\caption{PMCurv, PMConv and 95\% quantile of a Gaussian distribution with varying scale parameter.}
\label{fig:gauss2}
\end{figure}

\subsection{Student's $t$ distribution}

Figure~\ref{fig:studentt1} shows the Density function and its second derivative of a Student's $t$-distribution with 3 degrees of freedom. The points of inflection and maximum convexity are at $\pm \sqrt{\nu}/\sqrt{\nu+2}$ and $\pm \sqrt{3\nu}/\sqrt{\nu+2}$, respectively. The tails decay according to $x^{-\nu-1}$ and the kurtosis is $3+6/(\nu-4)$, when $\nu>4$. As $\nu$ decreases, the tail decays more slowly, and the kurtosis increases. Also, as $\nu$ decreases, the modal regions defined by $[\text{PInf}_l,\text{PInf}_r]$ and $[\text{PMConv}_l,\text{PMConv}_r]$ shrink, and the delimiting points for the tails get closer to the mode. On the other hand, the modal region defined by the 5\% and 95\% quantiles widens as the distributions become more heavy-tailed (see Figure \ref{fig:studentt2}). The previous inverse relationship between kurtosis and $\text{PInf}_r$ and $\text{PMConv}_r$ is present for many unimodal distributions, although an exact relationship is not straightforward to derive. For the Cauchy distribution, $\text{PMInf}_r$ and $\text{PMConv}_r$ correspond to the quantiles $2/3$ and $3/4$, respectively.

\begin{figure}[H]
   \centering
\includegraphics[width=\linewidth]{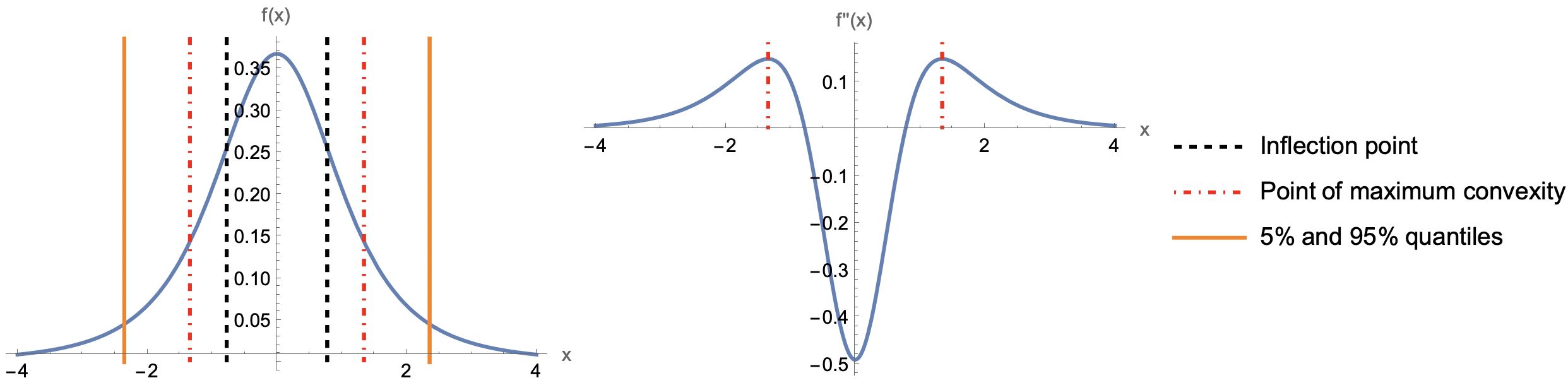}
\caption{Density function and its second derivative of a Student's $t$-distribution with 3 degrees of freedom. The black, red, and orange vertical lines are the $\text{PInf}$, $\text{PMConv}$, and the 5\% and 95\% quantiles, respectively.}
\label{fig:studentt1}
\end{figure}

\begin{figure}[H]
   \centering
\includegraphics[width=\linewidth]{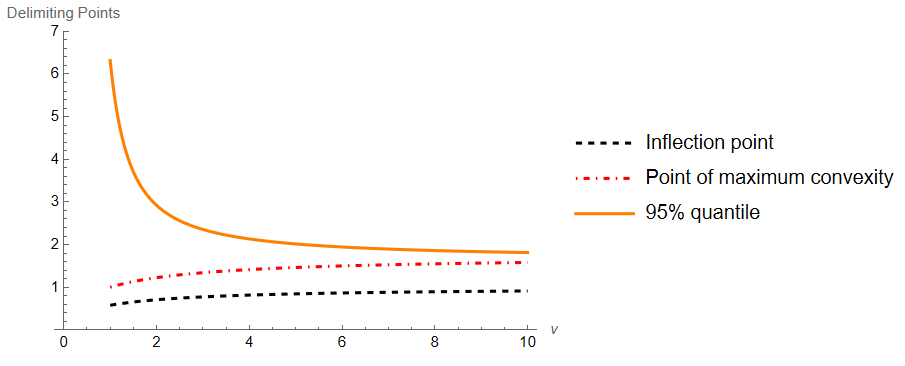}
\caption{Delimiting points for the right tail as a function of the degrees of freedom $\nu$ of a Student's $t$-distribution.}
\label{fig:studentt2}
\end{figure}

%\begin{figure}[H]
%   \centering
%\includegraphics[width=0.9\linewidth]{Images/pmc9.png}
%\caption{$\text{PMCurv}$ and 95\% quantile of a Student's t-distribution with varying degrees of freedom.}
%\end{figure}

\subsection{Skew-$t$ distribution}

%Skew-symmetric distributions constitute a broad set of probability distributions and have been a very active theme in distribution theory. A notable methodology for constructing such distributions was proposed by \cite{azzalini2005skew}, wherein the densities of skewed symmetric distributions are generated through the perturbation of a symmetric distribution density. In the univariate scenario, the density function takes the following form:
%$$
%\frac{2}{\sigma} f(z) F(s z), \quad z=\frac{x-\mu}{\sigma}.
%$$

%Building upon this construction, 
Figure~\ref{fig:skewt} shows the density function and its second derivative of a Skew-$t$ distribution with parameters $\nu=3$ and $s=10$. \cite{branco2001general} and \cite{azzalini2003distributions} define the density function of a skew-$t$ distribution as:

$$\frac{2}{\sigma} t(z ; \nu) T\left(s z \sqrt{\frac{\nu+1}{\nu+z^2}} ; \nu+1\right), \quad z=\frac{x-\mu}{\sigma},$$
where $t(x;\nu)$ and $T(x;\nu)$ are the pdf and cdf of a symmetric Student's $t$-distribution with $\nu$ degrees of freedom and scale parameter 1, where the parameter $s$ regulates the skewness. We compute the inflection points and point of maximum convexity numerically for $\mu=0$, $\sigma=1$, $s\in[-10,10]$, and $\nu \in [1,20]$. Figures~\ref{fig:skewt3} and \ref{fig:skewt2} illustrate that as $\nu$ and $s$ increase, $\text{PInf}_r$  and $\text{PMConv}_r$ move towards the center of the distribution. Interestingly, heightened levels of skewness bring the previous points closer to the mode. This behavior is also found for the Normal Inverse Gaussian distribution, which is a skewed and leptokurtic distribution \citep{barndorff1997normal}.

%This phenomenon may be related to the tail decay of the distribution.

\begin{figure}[H]
   \centering
\includegraphics[width=\linewidth]{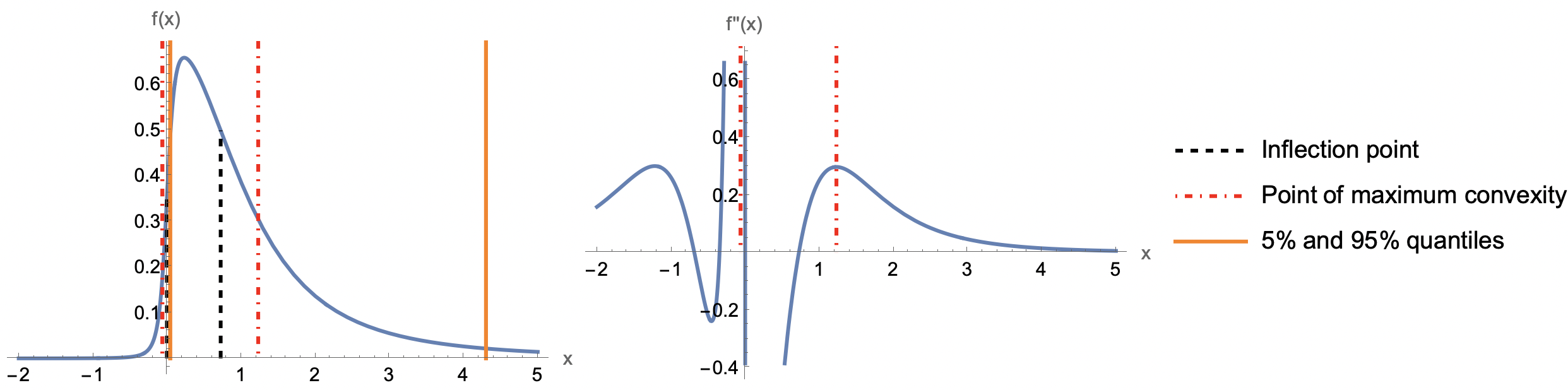}
\caption{Density function and its second derivative of a Skew-$t$ distribution with parameters $\nu=3$ and $s=10$. The black, red, and orange vertical lines are drawn in correspondence with the $\text{PInf}$, $\text{PMConv}$, and the 5\% and 95\% quantiles, respectively.}
\label{fig:skewt}
\end{figure}

\begin{figure}[H]
   \centering
\includegraphics[width=0.5\linewidth]{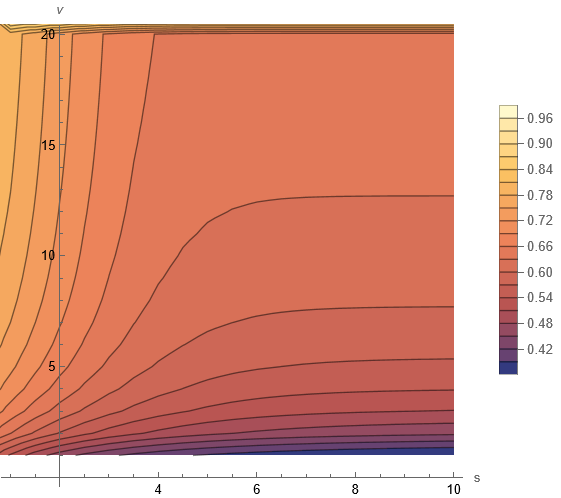}
\caption{Quantile corresponding to $\text{PInf}_r$ of a Skew-$t$ with parameters $\nu$ and $s$.}
\label{fig:skewt3}
\end{figure}

\begin{figure}[H]
   \centering
\includegraphics[width=0.5\linewidth]{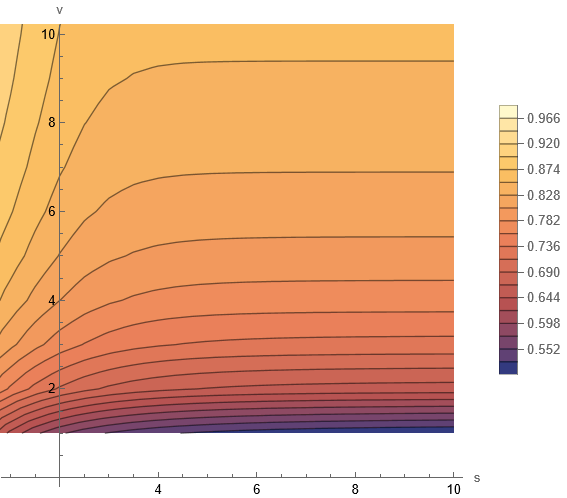}
\caption{Quantile corresponding to $\text{PConv}_r$ of a Skew-$t$ with parameters $\nu$ and $s$.}
\label{fig:skewt2}
\end{figure}

\begin{comment}
    
\subsection{Stable distribution}

Mention derivatives are unstable.

Third, for distributions with no closed-form expression for the pdf but a simple expression for the characteristic function, such as the $\alpha$-stable distribution, there is an efficient way of computing the cutoff points. By the simple application of the properties of derivatives of Fourier transforms it can be shown that PMConv is the root with respect to $x$ of $\int_{-\infty}^{\infty} \phi(t) t^3 e^{-i2\pi x t}dt$, and the inflection point is the root of $\int_{-\infty}^{\infty} \phi(t) t^2 e^{-i2\pi x t}dt$ with respect to $x$.

\begin{figure}[H]
   \centering
\includegraphics[width=\linewidth]{Images/alpha1.png}
\caption{Density function and its second derivative of a Stable distribution with parameters $\alpha=0.5$ and $\beta=0$, $\mu=0$, and $\sigma=1$. The black, red, and orange vertical lines are the $\text{PInf}$, $\text{PMConv}$, and the 5\% / 95\% quantiles, respectively.}
\label{fig:skewt}
\end{figure}

\begin{figure}[H]
   \centering
\includegraphics[width=\linewidth]{Images/alpha2.png}
\caption{Density function and its second derivative of a Stable distribution with parameters $\alpha=0.5$ and $\beta=0.5$, $\mu=0$, and $\sigma=1$. The black, red, and orange vertical lines are the $\text{PInf}$, $\text{PMConv}$, and the 5\% / 95\% quantiles, respectively.}
\label{fig:skewt}
\end{figure}
\end{comment}

\section{Discussion} \label{sect:discussion}

%and how does the inflection point, etc connect with the definitions of heavy taileddnss long tailedness etc
%For many common distributions, PMConv is increasing with kurtosis and decreasing with skewness. It will be interesting to draw out more precise relationships

This paper proposes using the inflection point, the point of maximum convexity (PMConv), and the point of maximum curvature (PMCurv) of probability density functions as the delimiting points between the bulk and the tail of the distribution. The inference of these delimiting points from the data is also discussed.

A further investigation into the relationships between these delimiting points and established measures of heavy-tailedness can provide valuable insights. Notably, it is observed that as kurtosis increases, the inflection point and point of maximum convexity tend to approach the mode and yield narrower modal regions. The previous observation holds for many common distributions—such as the Student's $t$-distribution, Inverse-Gaussian, Inverse-Gamma, and various sub-exponential distributions like the Log-Normal, Weibull, and Log-Gamma.   Conversely, the 5\% and 95\% quantiles generally lead to wider intervals when the kurtosis increases. 
The delimiting points for the right tail also tend to move closer to the mode with increasing skewness. This observation calls for further exploration into the precise connections between these statistical measures to offer a more nuanced understanding of distribution shapes and also how to extend these results to the multivariate case and multimodal distributions.

We also emphasize the importance of better understanding the properties of probability density function derivatives, a topic rarely touched upon in the literature (see, for instance, \citep{sato1981higher}). Particularly, for continuous and unimodal distributions, there appears to be a regular pattern in their derivatives,  which we exploit in the proof of Theorem~\ref{theo:1}.

%In contrast, the 5\% and 95\% quantiles tend to give intervals that widen when we increase the kurtosis.

\begin{appendices}

\section{Proof of Theorem \ref{theo:1}}  \label{theoproof:1}
\begin{proof}
      Since $\theta$ is the mode of the distribution we have $f'(\theta)=0$ and $f''(\theta)<0$.  We focus now on the shape of the probability function to the right of the mode. It is clear that $\lim_{|x|\to\infty}f^{(r)}(x)=0$ for $r=0,1,2$. 
    
    It then follows that there is a point $\omega>\theta$ such that $f''(\omega)>0$. If the previous assertion is false, then $f''(x)\leq 0$, for $x>\theta$, and consequently, since $f(x)$ is decreasing for $x>\theta$  ($f'(x)<0$ for $x>\theta$),  $\lim_{x\to\infty} f(x) = -\infty$, which contradicts the assumption that $f(x)$ is a density function. Now, since $f''(\theta)<0$, $f''(\omega)>0$, and $\lim_{x\to\infty}f''(x)=0$, then, due to the uniform continuity of $f''(x)$, there is a point $\theta<\text{PInf}_r<\omega$, such that $f''(\text{PInf}_r)=0$ (the second derivative changes sign), and $f''(x)$ possesses a $\text{PMConv}_r$, such that $f''(\text{PMConv}_r) = \max_{\theta<x<\infty}f''(x)$.

    A similar argument can be made for the existence of an inflection point and point of maximum curvature at the left of the mode.
\end{proof}

\section{Proof of Proposition \ref{prop:1}} \label{propproof:1}
\begin{proof} 
    We make use of the proof of Theorem \ref{theo:1}, and utilize the subscript $r$ in $\omega_r$, $\gamma_r$, and $\text{PMConv}_r$ to indicate that we are at the right of the mode, $x>\theta$. Since $f''(\theta)<0$, $f''(\text{PInf}_r)=0$, $f''(\omega_r)>0$, and $\theta<\text{PInf}_r<\omega_r$, then $f''(x)<0$ for $x\in \ ]\theta,\text{PInf}_r[$, and  $f''(x)>0$, for $x>\text{PInf}_r$. Therefore, since PMConv is in the interval $]\text{PInf}_r,\infty[$, then $\text{PMConv}_r>\text{PInf}_r$.     A similar argument can be used to show $\text{PMConv}_l<\text{PInf}_l$.
\end{proof}

\section{Proof of Theorem \ref{theo:2}}\label{theoproof:2}
\begin{proof}

We show next the asymptotic consistency of $f_n^{(r)}$, which was proven in \cite{schuster1969estimation} and will be required in this proof. If $f(x)$ and its first $r+1$ derivatives are bounded, then there exists positive constants $C_1$ and $C_2$ such that for every $\epsilon>0$,

\begin{equation}\label{eq:schuster}
    P\left\{\sup _x\left|f_n^{(r)}(x)-f^{(r)}(x)\right|>\epsilon\right\} \leq C_1 \exp \left(-C_2 n h_n^{2 r+2}\right),
\end{equation}
for sufficiently large $n$.

We now follow the proof in \cite{parzen1962estimation} related to the asymptotic consistency of the sample mode. Since $f^{(r)}(x)$ is uniformly continuous with an unique maximum then for every $\epsilon>0$ there exists an $\eta>0$ such that, for every point $x$, $|\theta^r-x|\geq\epsilon$ implies $|f(\theta^r)-f(x)| \geq \eta$. Thus, to prove $\theta_n^r$ converges to $\theta^r$ in probability, it suffices to show that for every $\epsilon > 0$,

\begin{equation} \label{eq:proof11}
    P(|f^{(r)}(\theta^r_n)-f^{(r)}(\theta^r)|>\epsilon) \to 0 \hspace{3cm} \text{as} \ n \to \infty.
\end{equation}

Now, since
$$
\left|f_n^{(r)}(\theta_n^r)-f(\theta^r)\right|=\left|\sup_x f_n^{(r)}(x)-\sup _x f^{(r)}(x)\right| \leq \sup_x\left|f_n^{(r)}(x)-f^{(r)}(x)\right|,
$$
then, by the triangular inequality, it follows that
\begin{align}
\left|f^{(r)}\left(\theta^r_n\right)-f^{(r)}(\theta^r)\right| &\leq \left|f^{(r)}(\theta^r_n)-f_n^{(r)}\left(\theta^r_n\right)\right|+| f_n^{(r)}(\theta^r_n)  -f^{(r)}(\theta^r)| \nonumber \\
 &\leq 2 \sup _x\left|f_n^{(r)}(x)-f^{(r)}(x)\right|. \nonumber
\end{align}
Finally, it follows that there is an upper of the probability in \eqref{eq:proof11}:
$$
P(|f^{(r)}(\theta^r_n)-f^{(r)}(\theta^r)|>\epsilon) \leq P\left(\sup _x\left|f_n^{(r)}(x)-f^{(r)}(x)\right|>\epsilon/2\right), 
$$
and from \ref{eq:schuster}, the probability on the right-hand side of the previous equation goes to 0 as $n\to\infty$.
\end{proof}

\end{appendices}
%intro add new applications, and change tone
%interesting to further study properties of the derivatives of density functions. only sparsely. ken ito sati paper
%estimation thing, add more references completeness, and run simulation example, prove convergence... ?
%add more distributions

\bibliography{bib.bib}

\begin{thebibliography}{29}
\providecommand{\natexlab}[1]{#1}
\providecommand{\url}[1]{\texttt{#1}}
\expandafter\ifx\csname urlstyle\endcsname\relax
  \providecommand{\doi}[1]{doi: #1}\else
  \providecommand{\doi}{doi: \begingroup \urlstyle{rm}\Url}\fi

\bibitem[Asmussen(2003)]{asmussen2003steady}
S{\o}ren Asmussen.
\newblock Steady-state properties of of gi/g/1.
\newblock \emph{Applied probability and Queues}, pages 266--301, 2003.

\bibitem[Azzalini and Capitanio(2003)]{azzalini2003distributions}
Adelchi Azzalini and Antonella Capitanio.
\newblock Distributions generated by perturbation of symmetry with emphasis on a multivariate skew t-distribution.
\newblock \emph{Journal of the Royal Statistical Society Series B: Statistical Methodology}, 65\penalty0 (2):\penalty0 367--389, 2003.

\bibitem[Balkema and De~Haan(1974)]{balkema1974residual}
August~A Balkema and Laurens De~Haan.
\newblock Residual life time at great age.
\newblock \emph{The Annals of probability}, 2\penalty0 (5):\penalty0 792--804, 1974.

\bibitem[Barndorff-Nielsen(1997)]{barndorff1997normal}
Ole~E Barndorff-Nielsen.
\newblock Normal inverse gaussian distributions and stochastic volatility modelling.
\newblock \emph{Scandinavian Journal of statistics}, 24\penalty0 (1):\penalty0 1--13, 1997.

\bibitem[Branco and Dey(2001)]{branco2001general}
M{\'a}rcia~D Branco and Dipak~K Dey.
\newblock A general class of multivariate skew-elliptical distributions.
\newblock \emph{Journal of Multivariate Analysis}, 79\penalty0 (1):\penalty0 99--113, 2001.

\bibitem[Cabral et~al.(2023{\natexlab{a}})Cabral, Bolin, and Rue]{cabral2023controlling}
Rafael Cabral, David Bolin, and H{\aa}vard Rue.
\newblock Controlling the flexibility of non-gaussian processes through shrinkage priors.
\newblock \emph{Bayesian Analysis}, 18\penalty0 (4):\penalty0 1223--1246, 2023{\natexlab{a}}.

\bibitem[Cabral et~al.(2023{\natexlab{b}})Cabral, Bolin, and Rue]{cabral2023fitting}
Rafael Cabral, David Bolin, and H{\aa}vard Rue.
\newblock Fitting latent non-gaussian models using variational bayes and laplace approximations.
\newblock \emph{Journal of the American Statistical Association}, \penalty0 (just-accepted):\penalty0 1--20, 2023{\natexlab{b}}.

\bibitem[Carvalho et~al.(2010)Carvalho, Polson, and Scott]{carvalho2010horseshoe}
Carlos~M Carvalho, Nicholas~G Polson, and James~G Scott.
\newblock The horseshoe estimator for sparse signals.
\newblock \emph{Biometrika}, 97\penalty0 (2):\penalty0 465--480, 2010.

\bibitem[Chac{\'o}n and Duong(2013)]{chacon2013data}
Jos{\'e}~E Chac{\'o}n and Tarn Duong.
\newblock Data-driven density derivative estimation, with applications to nonparametric clustering and bump hunting.
\newblock 2013.

\bibitem[do~Nascimento et~al.(2012)do~Nascimento, Gamerman, and Lopes]{do2012semiparametric}
Fernando~Ferraz do~Nascimento, Dani Gamerman, and Hedibert~Freitas Lopes.
\newblock A semiparametric bayesian approach to extreme value estimation.
\newblock \emph{Statistics and Computing}, 22:\penalty0 661--675, 2012.

\bibitem[Duong et~al.(2008)Duong, Cowling, Koch, and Wand]{duong2008feature}
Tarn Duong, Arianna Cowling, Inge Koch, and Matt~P Wand.
\newblock Feature significance for multivariate kernel density estimation.
\newblock \emph{Computational Statistics \& Data Analysis}, 52\penalty0 (9):\penalty0 4225--4242, 2008.

\bibitem[Foss et~al.(2011)Foss, Korshunov, Zachary, et~al.]{foss2011introduction}
Sergey Foss, Dmitry Korshunov, Stan Zachary, et~al.
\newblock \emph{An introduction to heavy-tailed and subexponential distributions}, volume~6.
\newblock Springer, 2011.

\bibitem[Guidoum(2020)]{guidoum2020kernel}
Arsalane~Chouaib Guidoum.
\newblock Kernel estimator and bandwidth selection for density and its derivatives: The kedd package.
\newblock \emph{arXiv preprint arXiv:2012.06102}, 2020.

\bibitem[H{\"a}rdle et~al.(1990)H{\"a}rdle, Marron, and Wand]{hardle1990bandwidth}
Wolfgang H{\"a}rdle, James~S Marron, and Matten~P Wand.
\newblock Bandwidth choice for density derivatives.
\newblock \emph{Journal of the Royal Statistical Society Series B: Statistical Methodology}, 52\penalty0 (1):\penalty0 223--232, 1990.

\bibitem[Huber(2004)]{huber2004robust}
Peter~J Huber.
\newblock \emph{Robust statistics}, volume 523.
\newblock John Wiley \& Sons, 2004.

\bibitem[Longin(2016)]{longin2016extreme}
Fran{\c{c}}ois Longin.
\newblock \emph{Extreme events in finance: A handbook of extreme value theory and its applications}.
\newblock John Wiley \& Sons, 2016.

\bibitem[MacDonald et~al.(2011)MacDonald, Scarrott, Lee, Darlow, Reale, and Russell]{macdonald2011flexible}
Anna MacDonald, Carl~John Scarrott, Dominic Lee, Brian Darlow, Marco Reale, and Glynn Russell.
\newblock A flexible extreme value mixture model.
\newblock \emph{Computational Statistics \& Data Analysis}, 55\penalty0 (6):\penalty0 2137--2157, 2011.

\bibitem[Parzen(1962)]{parzen1962estimation}
Emanuel Parzen.
\newblock On estimation of a probability density function and mode.
\newblock \emph{The annals of mathematical statistics}, 33\penalty0 (3):\penalty0 1065--1076, 1962.

\bibitem[Pickands~III(1975)]{pickands1975statistical}
James Pickands~III.
\newblock Statistical inference using extreme order statistics.
\newblock \emph{the Annals of Statistics}, pages 119--131, 1975.

\bibitem[Pisarenko and Rodkin(2010)]{pisarenko2010heavy}
V~Pisarenko and M~Rodkin.
\newblock \emph{Heavy-tailed distributions in disaster analysis}, volume~30.
\newblock Springer Science \& Business Media, 2010.

\bibitem[Politis et~al.(2015)Politis, Vasilev, Tarassenko, et~al.]{politis2015adaptive}
Dimitris~N Politis, Vyacheslav~A Vasilev, Peter~F Tarassenko, et~al.
\newblock Adaptive estimation of density function derivative.
\newblock 2015.

\bibitem[Rosenblatt(1956)]{rosenblatt1956remarks}
Murray Rosenblatt.
\newblock Remarks on some nonparametric estimates of a density function.
\newblock \emph{The annals of mathematical statistics}, pages 832--837, 1956.

\bibitem[Sato and Yamazato(1981)]{sato1981higher}
Ken-iti Sato and Makoto Yamazato.
\newblock On higher derivatives of distribution functions of class $ l$.
\newblock \emph{Journal of Mathematics of Kyoto University}, 21\penalty0 (3):\penalty0 575--591, 1981.

\bibitem[Scarrott and MacDonald(2012)]{scarrott2012review}
Carl Scarrott and Anna MacDonald.
\newblock A review of extreme value threshold estimation and uncertainty quantification.
\newblock \emph{REVSTAT-Statistical journal}, 10\penalty0 (1):\penalty0 33--60, 2012.

\bibitem[Schuster(1969)]{schuster1969estimation}
Eugene~F Schuster.
\newblock Estimation of a probability density function and its derivatives.
\newblock \emph{The Annals of Mathematical Statistics}, 40\penalty0 (4):\penalty0 1187--1195, 1969.

\bibitem[Siloko et~al.(2019)Siloko, Ikpotokin, Oyegue, Ishiekwene, and Afere]{siloko2019note}
IU~Siloko, O~Ikpotokin, FO~Oyegue, CC~Ishiekwene, and BAE Afere.
\newblock A note on application of kernel derivatives in density estimation with the univariate case.
\newblock \emph{Journal of Statistics and Management Systems}, 22\penalty0 (3):\penalty0 415--423, 2019.

\bibitem[Taleb(2020)]{taleb2020statistical}
Nassim~Nicholas Taleb.
\newblock Statistical consequences of fat tails: Real world preasymptotics, epistemology, and applications.
\newblock \emph{arXiv preprint arXiv:2001.10488}, 2020.

\bibitem[Tancredi et~al.(2006)Tancredi, Anderson, and O’Hagan]{tancredi2006accounting}
Andrea Tancredi, Clive Anderson, and Anthony O’Hagan.
\newblock Accounting for threshold uncertainty in extreme value estimation.
\newblock \emph{Extremes}, 9:\penalty0 87--106, 2006.

\bibitem[Teugels(1975)]{teugels1975class}
Jozef~L Teugels.
\newblock The class of subexponential distributions.
\newblock \emph{The Annals of Probability}, pages 1000--1011, 1975.

\end{thebibliography}
\end{document}